\documentclass[11pt,twoside]{amsart}
\usepackage{latexsym,amssymb,amsmath}

\numberwithin{equation}{section}
\def\demo{\noindent{\it Proof. }}
\newtheorem{theorem}{Theorem}[section]
\newtheorem{lemma}[theorem]{Lemma}
\newtheorem{proposition}[theorem]{Proposition}
\newtheorem{corollary}[theorem]{Corollary}
\newtheorem{conjecture}[theorem]{Conjecture}

\theoremstyle{definition}
\newtheorem{definition}[theorem]{Definition} 
 
\newtheorem{remark}[theorem]{Remark}
\newtheorem{example}[theorem]{Example}

\begin{document}


\title[Complete intersections]{Minimum distance functions of complete
intersections} 

\thanks{The first and third author were supported by SNI. The second author was
supported by CONACyT}
\author{Jos\'e Mart\'\i nez-Bernal}
\address{
Departamento de
Matem\'aticas\\
Centro de Investigaci\'on y de Estudios
Avanzados del
IPN\\
Apartado Postal
14--740 \\
07000 Mexico City, D.F.
}
\email{jmb@math.cinvestav.mx}

\author{Yuriko Pitones}
\address{
Departamento de
Matem\'aticas\\
Centro de Investigaci\'on y de Estudios
Avanzados del
IPN\\
Apartado Postal
14--740 \\
07000 Mexico City, D.F.
}
\email{ypitones@math.cinvestav.mx}

\author{Rafael H. Villarreal}
\address{
Departamento de
Matem\'aticas\\
Centro de Investigaci\'on y de Estudios
Avanzados del
IPN\\
Apartado Postal
14--740 \\
07000 Mexico City, D.F.
}
\email{vila@math.cinvestav.mx}

\begin{abstract} 
We study the {\it minimum distance function\/} of a complete
intersection graded ideal 
in a polynomial 
ring with coefficients in a field. For graded ideals of dimension
one, whose initial ideal is a complete
intersection, we use the footprint function to give a sharp 
lower bound for the minimum distance
function. 
Then we show some applications to coding theory.
\end{abstract}

\maketitle 

\section{Introduction}\label{intro-section}
Let $S=K[t_1,\ldots,t_s]=\oplus_{d=0}^{\infty} S_d$ be a polynomial ring over
a field $K$ with the standard grading and let $I\neq(0)$ be a graded ideal
of $S$.  The {\it degree\/} or {\it multiplicity\/} of $S/I$ is
denoted by $\deg(S/I)$. Fix a graded monomial order $\prec$ on $S$ 
and let ${\rm in}_\prec(I)$ be the initial ideal of $I$. 

The {\it footprint\/} of $S/I$ or {\it Gr\"obner \'escalier\/} of $I$, denoted
$\Delta_\prec(I)$, is the set of all monomials of $S$ not in the 
ideal ${\rm in}_\prec(I)$ \cite[p.~13, p.~133]{escalier}. This notion
occurs in other branches of
mathematics under different names; see \cite[p.~6]{cocoa-book} 
for a list of alternative names. 

Given an integer $d\geq 1$, let $\mathcal{M}_{\prec, d}$ be the set of 
all zero-divisors of $S/{\rm in}_\prec(I)$ of degree $d$ that
are in $\Delta_\prec(I)$, 
and let $\mathcal{F}_{\prec,d}$ be the set of all zero-divisors of $S/I$ that
are not zero and are a $K$-linear combination of monomials in $\Delta_\prec(I)$ of
degree $d$. 

The {\it footprint
function\/} of $I$, 
denoted ${\rm fp}_I$, is the function ${\rm fp}_I\colon
\mathbb{N}_+\rightarrow \mathbb{Z}$ given
by 
$$
{\rm fp}_I(d):=\left\{\begin{array}{ll}\deg(S/I)-\max\{\deg(S/({\rm
in}_\prec(I),t^a))\,\vert\,
t^a\in\mathcal{M}_{\prec, d}\}&\mbox{if }\mathcal{M}_{\prec, d}\neq\emptyset,\\
\deg(S/I)&\mbox{if }\mathcal{M}_{\prec, d}=\emptyset,
\end{array}\right.
$$
and the {\it minimum distance function\/} of $I$, denoted $\delta_I$,
is the function  
$\delta_I\colon \mathbb{N}_+\rightarrow \mathbb{Z}$ given by 
$$
\delta_I(d):=\left\{\begin{array}{ll}\deg(S/I)-\max\{\deg(S/(I,f))\vert\,
f\in\mathcal{F}_{\prec, d}\}&\mbox{if }\mathcal{F}_{\prec,d}\neq\emptyset,\\
\deg(S/I)&\mbox{if\ }\mathcal{F}_{\prec,d}=\emptyset.
\end{array}\right.
$$

These two functions were introduced and studied
in \cite{hilbert-min-dis}. Notice that $\delta_I$ is independent
of the monomial order $\prec$ (see Lemma~\ref{delta-indep-order}). 
To compute $\delta_I$ is a difficult problem but to compute 
${\rm fp}_I$ is much easier. 

We come to the main result of this paper which gives an explicit lower bound for 
$\delta_I$ and a formula for ${\rm fp}_I$ for a family of complete
intersection graded ideals: 

\noindent {\bf Theorem~\ref{footprint-ci}}{\it\   
If\, the initial ideal ${\rm in}_\prec(I)$ of $I$ is a complete intersection of
height $s-1$ generated by 
$t^{\alpha_2},\ldots,t^{\alpha_s}$, with $d_i=\deg(t^{\alpha_i})$ and 
$1\leq d_i\leq d_{i+1}$ for $i\geq 2$, then 
$$
\delta_I(d)\geq {\rm fp}_I(d)=\left\{\begin{array}{ll}(d_{k+2}-\ell)d_{k+3}\cdots d_s
&\mbox{if }\ 
d\leq \sum\limits_{i=2}^{s}\left(d_i-1\right)-1,\\ 
1 &\mbox{if\ }\ d\geq \sum\limits_{i=2}^{s}\left(d_i-1\right),
\end{array}\right.
$$
where $0\leq k\leq s-2$ and $\ell$ are integers such that 
$d=\sum_{i=2}^{k+1}\left(d_i-1\right)+\ell$ and $1\leq \ell \leq
d_{k+2}-1$.}

An important case of this theorem, from the viewpoint of
applications, 
is when $I$ is the vanishing 
ideal of a finite set of projective points over a finite
field (see the discussion below about the connection of 
 ${\rm fp}_I$ and $\delta_I$ with coding theory). If $I$ is a 
complete intersection monomial ideal of dimension $1$, then
$\delta_I(d)={\rm fp}_I(d)$ for $d\geq 1$ (see
Proposition~\ref{geil-carvalho-monomial}), but this case is only of 
theoretical interest because, by
Proposition~\ref{classification-vanishing-monomial}, 
a monomial ideal is a vanishing ideal only in
particular cases.  

Let $I\subset S$ be a graded ideal such that $L={\rm in}_\prec(I)$ is a 
complete intersection of dimension $1$. We give a formula for the 
degree  of $S/(L,t^a)$ when $t^a$ is in $\mathcal{M}_{\prec, d}$, that
is, $t^a$ is not in $L$ and is a zero-divisor of $S/L$. By an easy
classification of the complete intersection property of $L$ (see
Lemma~\ref{ci-monomial-dim1})  
there are basically two cases to consider. One of
them is Lemma~\ref{dec28-15}, and the other is the following: 

\noindent {\bf Lemma~\ref{dec20-15}}{\it\  If $L={\rm in}_\prec(I)$ is generated by 
$t_2^{d_2},\ldots ,t_s^{d_s}$ and $t^a=t_1^{a_1}\cdots
t_s^{a_s}$ is in $\mathcal{M}_{\prec, d}$, then  
$$
\deg(S/(L,t^a))=d_2\cdots d_s- (d_2-a_2)\cdots(d_s-a_s).
$$
}
\quad To show our main result we use the formula for the degree of
the ring $S/(L,t^a)$, and then use Proposition~\ref{aug-28-15} to
bound the degrees uniformly. The proof of the main result takes place
in an abstract algebraic setting with no reference to vanishing
ideals or finite fields.  

The formulas for the degree are useful in the following setting.  
If $I=I(\mathbb{X})$ is the vanishing ideal of a finite set
$\mathbb{X}$ of
projective points, and ${\rm in}_\prec(I)$ is generated by
$t_2^{d_2},\ldots,t_s^{d_s}$, then Lemma~\ref{dec20-15} can be used to
give  upper bounds for
the number of zeros in $\mathbb{X}$ of homogeneous polynomials of $S$. In fact, if  
$f\in\mathcal{F}_{\prec,d}$  and 
${\rm in}_\prec(f)=t_1^{a_1}\cdots t_s^{a_s}$, then ${\rm
in}_\prec(f)$ is in $\mathcal{M}_{\prec,d}$, and by
Corollary~\ref{bounds-for-deg-init-ci-case-1} one has:  
$$
|V_{\mathbb{X}}(f)|\leq d_2\cdots d_s- (d_2-a_2)\cdots(d_s-a_s),
$$
where $V_\mathbb{X}(f)$ is the set of zeros or variety 
of $f$ in $\mathbb{X}$. 
This upper bound depends on the
exponent of the leading term of $f$. A more complex upper bound 
is obtained when the initial ideal of $I(\mathbb{X})$ is as 
in Lemma~\ref{ci-monomial-dim1}(ii). In this case one uses
the formula for the degree given in 
Lemma~\ref{dec28-15}.

The interest in studying ${\rm fp}_I$ and $\delta_I$ comes from
algebraic coding theory. 
Indeed, if $I=I(\mathbb{X})$ is the vanishing
ideal of a finite subset $\mathbb{X}$ of a projective space $\mathbb{P}^{s-1}$ over a
finite field $K=\mathbb{F}_q$, then the minimum distance
$\delta_\mathbb{X}(d)$ 
of the corresponding projective Reed-Muller-type code 
is  equal to $\delta_{I(\mathbb{X})}(d)$, 
and ${\rm fp}_{I(\mathbb{X})}(d)$ is a lower bound for
$\delta_\mathbb{X}(d)$ for $d\geq 1$ (see Theorem~\ref{min-dis-vi}
and Lemma~\ref{dec30-15}). Therefore, one has the formula: 
$$
\delta_{I(\mathbb{X})}(d)=\deg(S/I(\mathbb{X}))-\max\{|V_{\mathbb{X}}(f)|\colon
f\not\equiv 0;
f\in S_d\},\leqno(\ref{min-dis-vi-coro})
$$
where $f\not\equiv 0$ means that $f$ is not 
the zero function on $\mathbb{X}$. Our abstract study of the minimum
distance and footprint
functions provides fresh techniques to study $\delta_\mathbb{X}(d)$. 

It is well-known that the degree of $S/I(\mathbb{X})$ is equal
to $|\mathbb{X}|$ \cite[Lecture 13]{harris}. 
Hence, using Eq.~(\ref{min-dis-vi-coro}) and our main 
result, we get the following uniform upper bound for the number of
zeros of all polynomials $f\in S_d$ that do not vanish at all points of
$\mathbb{X}$.

\noindent {\bf Corollary~\ref{uniform-bound-for-zeros}}{\it\  
If\, the initial ideal ${\rm in}_\prec(I(\mathbb{X}))$ is a complete
intersection generated by 
$t^{\alpha_2},\ldots,t^{\alpha_s}$, with $d_i=\deg(t^{\alpha_i})$ and 
$1\leq d_i\leq d_{i+1}$ for $i\geq 2$, then
$$
|V_\mathbb{X}(f)|\leq 
|\mathbb{X}|-\left(d_{k+2}-\ell\right)d_{k+3}\cdots
d_s,\leqno(\ref{eq-unif-b})
$$
for any $f\in S_d$ that does not vanish at all point of $\mathbb{X}$,
where $0\leq k\leq s-2$ and $\ell$ are integers such that 
$d=\sum_{i=2}^{k+1}\left(d_i-1\right)+\ell$ and $1\leq \ell \leq
d_{k+2}-1$.}

This result gives a tool for finding good uniform upper bounds for
the number of zeros in $\mathbb{X}$ of polynomials over 
finite fields. This is a problem of fundamental interest in 
algebraic coding theory \cite{sorensen} and 
algebraic geometry \cite{schmidt}.  We leave as an open question 
whether this uniform bound is optimal, that is, 
whether the equality is attained for some polynomial $f$.

Toh\v{a}neanu and Van Tuyl 
conjectured \cite[Conjecture~4.9]{tohaneanu-vantuyl} that if the
vanishing ideal $I(\mathbb{X})$ is a complete intersection 
generated by polynomials of degrees
$d_2,\ldots,d_s$ and $d_i\leq d_{i+1}$ for all $i$, then 
$\delta_\mathbb{X}(1)\geq (d_2-1)d_3\ldots d_s$. By 
Corollary~\ref{uniform-bound-for-zeros} this conjecture is true if 
${\rm in}_\prec(I(\mathbb{X}))$ is a complete intersection. We leave
as another open question whether Corollary~\ref{uniform-bound-for-zeros} is
true if we only assume that $I(\mathbb{X})$ is a complete intersection
(cf. Proposition~\ref{jan2-16}).

To illustrate the use of 
Corollary~\ref{uniform-bound-for-zeros} in a concrete situation,  
consider the lexicographical order on $S$ with $t_1\prec\cdots\prec
t_s$ and a {\it projective torus\/} over a 
finite field $\mathbb{F}_q$ with $q\neq 2$: 
$$
\mathbb{T}=\{[(x_1,\ldots,x_s)]\in\mathbb{P}^{s-1}\vert\, x_i\in
\mathbb{F}_q^*\, \mbox{ for }i=1,\ldots,s\},
$$
where $\mathbb{F}_q^*=\mathbb{F}_q\setminus\{0\}$. 
As $I(\mathbb{T})$ is generated by the Gr\"obner basis 
$\{t_i^{q-1}-t_1^{q-1}\}_{i=2}^s$, its initial ideal is a complete
intersection generated by 
$t_2^{q-1},\ldots,t_s^{q-1}$. Therefore, noticing that 
$\deg(S/I(\mathbb{T}))$ is equal to $(q-1)^{s-1}$ and setting $d_i=q-1$ for
$i=2,\ldots,s$ in Eq.~(\ref{eq-unif-b}), we obtain that any  
homogeneous polynomial $f$ of degree 
$d$, not vanishing at all points of $\mathbb{T}$, has at most
$$
(q-1)^{s-1}-(q-1)^{s-(k+2)}(q-1-\ell) 
$$
zeros in $\mathbb{T}$ if $d\leq (q-2)(s-1)-1$, and $k$ and $\ell$ are 
the unique integers such that $k\geq 0$,
$1\leq \ell\leq q-2$ and $d=k(q-2)+\ell$. This uniform bound was given
in \cite[Theorem 3.5]{ci-codes} and it is seen that 
this bound is in fact optimal by constructing an appropriate
polynomial $f$. 

If ${\rm fp}_I(d)=\delta_I(d)$ for $d\geq 1$, we say that $I$ is a
{\it Geil--Carvalho ideal\/}. For vanishing ideals over finite fields,
this notion is essentially another way of saying that the bound of
Eq.~(\ref{eq-unif-b}) is optimal. The first interesting family of ideals 
where equality holds is due to Geil \cite[Theorem~2]{geil}. 
His result essentially shows that 
${\rm fp}_I(d)=\delta_I(d)$ for $d\geq 1$ when $\prec$ is a graded
lexicographical order and $I$ is the 
homogenization of the vanishing ideal of the affine space
$\mathbb{A}^{s-1}$ over a finite field $K=\mathbb{F}_q$. Recently, 
Carvalho \cite[Proposition~2.3]{carvalho} extended this result by
replacing $\mathbb{A}^{s-1}$ by a cartesian product of subsets 
of $\mathbb{F}_q$. In this case the underlying Reed-Muller-type code
is called an affine cartesian code and an explicit formula for the
minimum distance was first given in 
\cite{geil-thomsen,cartesian-codes}. In a very recent paper, Bishnoi,
Clark, Potukuchi, and Schmitt give another proof of this formula 
\cite[Theorem~5.2]{clark} using a result of Alon and 
F\"uredi \cite[Theorem~5]{alon-furedi} (see also \cite{clark1}). 

As the two most relevant applications of our main result to algebraic
coding
theory, we recover 
the formula
for the minimum distance of an affine cartesian code given in
\cite[Theorem~3.8]{cartesian-codes} and
\cite[Proposition~5]{geil-thomsen}, 
and the fact that the
homogenization of the corresponding vanishing ideal is a 
Geil--Carvalho ideal \cite{carvalho} (see
Corollary~\ref{lopez-renteria-vila}). 

Then we present an extension of a
result of Alon and F\"uredi 
\cite[Theorem~1]{alon-furedi}---in terms of the regularity of a
vanishing ideal---about coverings of the cube 
$\{0,1\}^n$ by affine hyperplanes, that can be applied to any finite subset 
of a projective space whose vanishing ideal has a complete intersection
initial ideal (see Corollary~\ref{jan1-16} and
Example~\ref{covering-by-hyperplanes-example}). 

Finally, using {\em Macaulay\/}$2$ \cite{mac2}, 
we exemplify how some of our results
can be used in practice, and show that 
the vanishing ideal of $\mathbb{P}^2$ over $\mathbb{F}_2$ 
is not Geil--Carvalho by computing all possible initial ideals (see
Example~\ref{not-geil-carvalho}).  

In Section~\ref{prelim-section} we 
introduce projective Reed-Muller-type codes and 
present some of the results and terminology that will be needed
in the paper. For all unexplained
terminology and additional information,  we refer to 
\cite{Vogel} (for deeper advances on the knowledge of the degree), 
\cite{CLO} (for the theory of Gr\"obner bases), \cite{AM,Eisen,Sta1} 
(for commutative algebra and Hilbert functions), 
and \cite{MacWilliams-Sloane,tsfasman} (for the theory of
error-correcting codes and linear codes).

\section{Preliminaries}\label{prelim-section}

In this section, we 
present some of the results that will be needed throughout the paper
and introduce some more notation. All results of this
section are well-known. 

Let $S=K[t_1,\ldots,t_s]=\oplus_{d=0}^{\infty}S_d$ be a graded
polynomial ring over a field
$K$ with the standard grading and let $(0)\neq I\subset S$ be a
graded ideal.   
The {\it Hilbert function} of $S/I$ is: 
$$
H_I(d):=\dim_K(S_d/I_d),\ \ \ d=0,1,2,\ldots
$$
where $I_d=I\cap S_d$. By the dimension
of $I$ we mean the Krull dimension of $S/I$.  

The {\it degree\/} or {\it multiplicity\/} of $S/I$ is the 
positive integer 
$$
\deg(S/I):=\left\{\begin{array}{ll}(k-1)!\,\displaystyle\lim_{d\rightarrow\infty}{H_I(d)}/{d^{k-1}}
&\mbox{if }k\geq 1,\\
\dim_K(S/I) &\mbox{if\ }k=0,
\end{array}\right.
$$ 
and the {\it regularity of the Hilbert function\/} of $S/I$, or simply the
\emph{regularity} of $S/I$, denoted 
${\rm reg}(S/I)$, is the least integer $r\geq 0$ such that
$H_I(d)$ is equal to $h_I(d)$ for $d\geq r$, where $h_I$ is the
Hilbert polynomial of $S/I$. 

Let  $\prec$ be a monomial order on $S$ and let $(0)\neq I\subset S$
be an ideal. The {\it leading monomial\/} of $f$ 
is denoted by ${\rm in}_\prec(f)$ and the {\it initial ideal\/} of $I$
is denoted by ${\rm in}_\prec(I)$. A monomial $t^a$ is called a 
{\it standard monomial\/} of $S/I$, with respect 
to $\prec$, if $t^a$ is not in the ideal 
${\rm in}_\prec(I)$. A polynomial $f$ is called {\it standard\/} if
$f\neq 0$ and $f$ is a
$K$-linear combination of standard monomials. The set of standard
monomials, denoted $\Delta_\prec(I)$, is called the {\it
footprint\/} of $S/I$. 
If $I$ is graded, then $H_I(d)$ is the number of standard
monomials of degree $d$. 

\begin{lemma}\label{regular-elt-in} 
Let $\prec$ be a monomial order, let $I\subset S$ be an ideal, and let
$f$ be a polynomial of $S$ of positive degree. If ${\rm in}_\prec(f)$
is regular on $S/{\rm in}_\prec(I)$, then $f$ is regular on $S/I$. 
\end{lemma}

\begin{proof} Let $g$ be a polynomial of $S$ such that $gf\in I$. It
suffices to show that $g\in I$. Pick a Gr\"obner basis
$g_1,\ldots,g_r$ of $I$. Then, by the division algorithm
\cite[Theorem~3, p.~63]{CLO}, we can write 
$g=\sum_{i=1}^ra_ig_i+h$, 
where $h=0$ or $h$ is a standard polynomial of $S/I$. We need only
show that $h=0$. If $h\neq 0$, then $hf$ is in $I$ and 
${\rm in}_\prec(h){\rm in}_\prec(f)$ is in ${\rm
in}_\prec(I)$. Therefore ${\rm in}_\prec(h)$ is in ${\rm
in}_\prec(I)$, a contradiction.
\end{proof}

\begin{remark} Given an integer $d\geq 1$, there is a map 
${\rm in}_\prec\colon\mathcal{F}_{\prec,d}\rightarrow
\mathcal{M}_{\prec, d}$ given by $f\mapsto {\rm
in}_\prec(f)$. This follows from Lemma~\ref{regular-elt-in}. If $I$ is
a monomial ideal, then $\mathcal{M}_{\prec, d}\subset
\mathcal{F}_{\prec, d}$.
\end{remark}

\paragraph{\bf Projective Reed-Muller-type codes} 
Let $K=\mathbb{F}_q$ be a finite field with $q$ elements,
let $\mathbb{P}^{s-1}$ be a projective space over 
$K$, and let $\mathbb{X}$ be a subset of
$\mathbb{P}^{s-1}$. As usual, points of $\mathbb{P}^{s-1}$ are denoted by
$[\alpha]$, where $0\neq\alpha\in K^s$.  In this paragraph all results
are valid if we assume that $K$ is any field and $\mathbb{X}$ is a finite subset of
$\mathbb{P}^{s-1}$, instead of assuming that $K$ is finite. However, the
interesting case for coding theory is when $K$ is finite.

The {\it vanishing ideal\/} of $\mathbb{X}$, denoted $I(\mathbb{X})$, is the ideal of $S$ 
generated by the homogeneous polynomials that vanish at all points of
$\mathbb{X}$. In this case the Hilbert function of $S/I(\mathbb{X})$ is denoted by
$H_\mathbb{X}(d)$. 
Let $P_1,\ldots,P_m$ be a set of
representatives for the points of $\mathbb{X}$ with $m=|\mathbb{X}|$.
Fix a degree $d\geq 1$. 
For each $i$ there is $f_i\in S_d$ such that
$f_i(P_i)\neq 0$. Indeed suppose $P_i=[(a_1,\ldots,a_s)]$, there is at
least one $k$ in $\{1,\ldots,s\}$ such that $a_k\neq 0$. Setting
$f_i(t_1,\ldots,t_s)=t_k^d$ one has that $f_i\in S_d$ and
$f_i(P_i)\neq 0$. There is a $K$-linear map: 
\begin{equation}\label{ev-map}
{\rm ev}_d\colon S_d=K[t_1,\ldots,t_s]_d\rightarrow K^{|\mathbb{X}|},\ \ \ \ \ 
f\mapsto
\left(\frac{f(P_1)}{f_1(P_1)},\ldots,\frac{f(P_m)}{f_m(P_m)}\right).
\end{equation}

The map ${\rm ev}_d$ is called an {\it evaluation map}. The image of 
$S_d$ under ${\rm ev}_d$, denoted by  $C_\mathbb{X}(d)$, is called a {\it
projective Reed-Muller-type code\/} of degree $d$ over $\mathbb{X}$ 
\cite{duursma-renteria-tapia}. 
It is also called an {\it evaluation code\/} associated to $\mathbb{X}$
\cite{gold-little-schenck}. This type of codes have been studied using
commutative algebra methods and especially Hilbert functions, 
see \cite{delsarte-goethals-macwilliams,GRT,algcodes,sorensen} and the
references therein. 

\begin{definition}\rm A {\it linear
code\/} is a linear subspace of $K^m$ for some 
$m$. The {\it basic parameters} of the linear
code $C_\mathbb{X}(d)$ are its {\it length\/} $|\mathbb{X}|$, {\it
dimension\/} 
$\dim_K C_\mathbb{X}(d)$, and {\it minimum distance\/} 
$$
\delta_\mathbb{X}(d):=\min\{\|v\| 
\colon 0\neq v\in C_\mathbb{X}(d)\},
$$
where $\|v\|$ is the number of non-zero
entries of $v$. 
\end{definition}

\begin{lemma}{\cite[Lemma~2.13]{hilbert-min-dis}}\label{may21-15-1}
{\rm (a)} The map ${\rm ev}_d$ is well-defined, i.e., it is independent of
the set of representatives that we choose for the points of
$\mathbb{X}$. {\rm (b)} The basic parameters of the Reed-Muller-type code
$C_\mathbb{X}(d)$ are
independent of $f_1,\ldots,f_m$.
\end{lemma}

The following summarizes the well-known relation between 
projective Reed-Muller-type codes and the theory of Hilbert 
functions. Notice that items (i) and (iv) follow directly from
Eq.~(\ref{ev-map}) and 
item (iii), respectively.

\begin{proposition}\label{jan4-15}
The following hold. 
\begin{itemize}
\item[{\rm (i)}] $H_\mathbb{X}(d)=\dim_KC_\mathbb{X}(d)$ for $d\geq 1$.
\item[{\rm(ii)}] {\rm \cite[Lecture 13]{harris}} $\deg(S/I(\mathbb{X}))=|\mathbb{X}|$. 
\item[{\rm (iii)}] {\rm(}Singleton bound{\rm)} $1\leq
\delta_\mathbb{X}(d)\leq |\mathbb{X}|-H_\mathbb{X}(d)+1$ for $d\geq 1$. 
\item[{\rm (iv)}] $\delta_\mathbb{X}(d)=1$ for $d\geq {\rm reg}(S/I(\mathbb{X}))$.  
\end{itemize}
\end{proposition}

The next result gives an algebraic
formulation of the minimum distance of a projective Reed-Muller-type code 
in terms of the degree and the structure of the underlying
vanishing ideal.

\begin{theorem}{\cite[Theorem~4.7]{hilbert-min-dis}}\label{min-dis-vi} If $|\mathbb{X}|\geq 2$, then
$\delta_\mathbb{X}(d)=\delta_{I(\mathbb{X})}(d)\geq 1$ 
for $d\geq 1$.
\end{theorem}

This result gives an algorithm, that can be implemented in
\emph{CoCoA} \cite{CNR}, {\em Macaulay\/}$2$
\cite{mac2}, or \emph{Singular} \cite{singular}, 
to compute $\delta_\mathbb{X}(d)$ for small values of $q$
and $s$, where $q$ is the cardinality of $\mathbb{F}_q$ and $s$ is the
number of variables of $S$ (see the procedure of
Example~\ref{covering-by-hyperplanes-example}). 
Using SAGE \cite{sage} one can also compute $\delta_\mathbb{X}(d)$ by 
finding a generator matrix of $C_\mathbb{X}(d)$.

As a direct consequence of Theorem~\ref{min-dis-vi} one has:
\begin{equation}\label{min-dis-vi-coro}
\delta_{I(\mathbb{X})}(d)=\deg(S/I(\mathbb{X}))-\max\{|V_{\mathbb{X}}(f)|\colon
f\not\equiv 0;
f\in S_d\},
\end{equation}
where $V_\mathbb{X}(f)$ is the zero set of $f$ in
$\mathbb{X}$ 
and $f\not\equiv 0$ means that $f$ does not vanish at all points of 
$\mathbb{X}$. 

The next lemma follows using 
the division algorithm \cite{CLO} (cf. \cite[Problem~1-17]{fulton}).   

\begin{lemma}\label{primdec-ix-a} Let $\mathbb{X}$ be a finite
subset of $\mathbb{P}^{s-1}$, let $[\alpha]$ be a point in
$\mathbb{X}$, 
with $\alpha=(\alpha_1,\ldots,\alpha_s)$
and $\alpha_k\neq 0$ for some $k$, and let
$I_{[\alpha]}$ be the vanishing ideal of $[\alpha]$. Then $I_{[\alpha]}$ is a prime ideal, 
\begin{equation*}
I_{[\alpha]}=(\{\alpha_kt_i-\alpha_it_k\vert\, k\neq i\in\{1,\ldots,s\}),\
\deg(S/I_{[\alpha]})=1,\,  
\end{equation*}
${\rm ht}(I_{[\alpha]})=s-1$, 
and $I(\mathbb{X})=\bigcap_{[\beta]\in{\mathbb{X}}}I_{[\beta]}$ is the primary
decomposition of $I(\mathbb{X})$. 
\end{lemma}

\begin{remark}\label{jun8-16} If $\mathbb{X}$ is a finite set of projective points,
then $S/I(\mathbb{X})$ is a Cohen--Macaulay reduced graded ring of dimension
$1$. This is very well-known, and it follows directly from
Lemma~\ref{primdec-ix-a}. In particular, the regularity of 
the Hilbert function of $S/I(\mathbb{X})$ is the Castelnuovo--Mumford regularity of
$S/I(\mathbb{X})$. 
\end{remark} 

An ideal $I\subset S$ is called {\it unmixed\/} 
if all its associated primes have the same height. The next result classifies monomial 
vanishing ideals of
finite sets in a projective space.

\begin{proposition}\label{classification-vanishing-monomial} 
Let $\mathbb{X}$ be a finite subset of $\mathbb{P}^{s-1}$.  The
following are equivalent\/{\rm:}
\begin{itemize}
\item[\rm(a)] $I(\mathbb{X})$ is a monomial ideal.
\item[\rm(b)] $I(\mathbb{X})=\cap_{i=1}^m\mathfrak{p}_i$, where 
the $\mathfrak{p}_i$'s are generated by $s-1$ variables. 
\item[\rm(c)] $\mathbb{X}\subset\{[e_1],\ldots,[e_s]\}$, where $e_i$
is the $i$-th unit vector.
\end{itemize}
\end{proposition}

\begin{proof} (a) $\Rightarrow$ (b): By Remark~\ref{jun8-16}, 
$I(\mathbb{X})$ is a radical
Cohen--Macaulay graded ideal of dimension $1$. Hence, 
$I(\mathbb{X})$ is an unmixed square-free monomial ideal of height
$s-1$. Therefore, $I(\mathbb{X})$ is equal to $\cap_{i=1}^m\mathfrak{p}_i$, where 
the $\mathfrak{p}_i$'s are face ideals (i.e., ideals generated by
variables) of height $s-1$.  

(b) $\Rightarrow$ (c): Let $\overline{\mathbb{X}}$ be the Zariski
closure of $\mathbb{X}$. As $\mathbb{X}$ is finite, one has 
$\mathbb{X}=\overline{\mathbb{X}}=V(I(\mathbb{X}))=\cup_{i=1}^mV(\mathfrak{p}_i)$. Thus it
suffices to notice that $V(\mathfrak{p}_i)=\{[e_{i_j}]\}$ for some
$i_j$. 

(c) $\Rightarrow$ (a): This follows from Lemma~\ref{primdec-ix-a}. 
\end{proof}

\section{Complete intersections}\label{ci-section}

Let $S=K[t_1,\ldots,t_s]=\oplus_{d=0}^{\infty} S_d$ be a polynomial ring over a field $K$ with
the standard grading and $s\geq 2$. An ideal $I\subset S$ is called a
{\it complete intersection\/} if  
there exist $g_1,\ldots,g_{r}$ in $S $ such that $I=(g_1,\ldots,g_{r})$, 
where $r$ is the height of $I$.

In what follows by a monomial order $\prec$
we mean a graded monomial order in the sense that $\prec$ is defined
first by total degree \cite{CLO}.

\begin{lemma}\label{ci-monomial-dim1} Let $L\subset S$ be an ideal
 generated by monomials. If $\dim(S/L)=1$,
then $L$ is a complete intersection if and only if, up to permutation of
the variables $t_1,\ldots,t_s$, we can write
\begin{itemize}
\item[(i)] $L=(t_2^{d_2},\ldots,t_s^{d_s})$ with $1\leq d_i\leq d_{i+1}$ for 
$i\geq 2$, or 
\item[(ii)] $L=(t_1^{d_2},\ldots,t_{p-1}^{d_p},t_p^{c_p}t_{p+1}^{c_{p+1}},t_{p+2}^{d_{p+2}},\ldots
,t_s^{d_s})$ for some $p\geq 1$ such that $1\leq c_p\leq c_{p+1}$ and
$1\leq d_i\leq d_{i+1}$ for $2\leq i\leq s-1$, where
$d_{p+1}=c_p+c_{p+1}$.
\end{itemize}
\end{lemma}

\begin{proof} $\Rightarrow$) Let
$t^{\alpha_1},\ldots,t^{\alpha_{s-1}}$ be the minimal set of
generators of $L$ consisting of monomials. These monomials form a
regular sequence. Hence $t^{\alpha_i}$ and
$t^{\alpha_j}$ have no common variables for $i\neq j$. Then, either all
variables occur in $t^{\alpha_1},\ldots,t^{\alpha_{s-1}}$ and we are
in case (ii), up to permutation of the variables $t_1,\ldots,t_s$, or there is one 
variable that is not in any of the
$t^{\alpha_i}$'s and we are in case (i), up to permutation of
the variables $t_1,\ldots,t_s$. 

$\Leftarrow$) In both cases $L$ is an ideal of height $s-1$ generated
by $s-1$ elements, that is, $L$ is a complete intersection.
\end{proof}

\begin{proposition}{\cite[Propositions~3.1.33 and 5.1.11]{monalg-rev}}\label{inter-tensoraa}
Let $A=R_1/I_1$, $B=R_2/I_2$ be two standard graded algebras over a field
$K$, where $R_1=K[\mathbf{x}]$, $R_2=K[\mathbf{y}]$ are polynomial 
rings in disjoint sets of variables and $I_i$ is an ideal 
of $R_i$. If $R=K[\mathbf{x},\mathbf{y}]$ and $I=I_1R+I_2R$, then 
$$
(R_1/I_1)\otimes_K(R_2/I_2)\simeq R/I\
\mbox{ and }\ F(A\otimes_K B,x)=F(A,x)F(B,x),
$$
where $F(A,x)$ and $F(B,x)$ are the Hilbert series of $A$ and $B$,
respectively.
\end{proposition}

\begin{lemma}\label{dec20-15} Let $L$ be the ideal of $S$ generated by 
$t_2^{d_2},\ldots ,t_s^{d_s}$. If $t^a=t_1^{a_1}t_r^{a_r}\cdots
t_s^{a_s}$, $r\geq 2$, $a_r\geq 1$, and $a_i\leq d_i-1$ for $i\geq r$,
then 
$$
\deg(S/(L,t^a)) =\deg(S/(L,t_r^{a_r}\cdots
t_s^{a_s}))= d_2\cdots d_s- d_2\cdots d_{r-1}(d_r-a_r)\cdots(d_s-a_s),
$$
where $a_i=0$ if $2\leq i<r$.
\end{lemma}

\demo In what follows we will use the fact that Hilbert
functions and Hilbert series are additive on short exact sequences
\cite[Chapter~2, Proposition~7]{fulton}.  
If $a_1\geq 1$, then taking Hilbert
functions in the exact sequence
$$
0\longrightarrow
S/(L,t_r^{a_r}\cdots
t_s^{a_s})[-a_1]
\stackrel{t_1^{a_1}}{\longrightarrow} S/(L,t^a)\longrightarrow
S/(L,t_1^{a_1})\longrightarrow 0, 
$$
and noticing that $\dim(S/(L,t_1^{a_1}))=0$, the first equality follows. Thus we may assume that 
$t^a$ has the form $t^a=t_r^{a_r}\cdots t_s^{a_s}$ and $a_i=0$ for $i<r$.

We proceed by induction on $s\geq 2$. Assume $s=2$. Then $r=2$, $t^a=t_2^{a_2}$,
$(L,t^a)=(t_2^{a_2})$, and the degree of $S/(L,t^a)$ is $a_2$, as
required. Assume $s\geq 3$. If $a_i=0$ for $i>r$, then
$(L,t^a)=(L,t_r^{a_r})$ is a complete intersection and the required
formula follows from \cite[Corollary~3.3]{Sta1}. Thus we may assume
that $a_i\geq 1$ for some $i>r$. There is an exact sequence 
\begin{eqnarray}\label{dec18-15}
& & 0\longrightarrow
S/(t_2^{d_2},\ldots,t_{r-1}^{d_{r-1}},t_r^{d_r-a_r},t_{r+1}^{d_{r+1}},\ldots,t_s^{d_s},
t_{r+1}^{a_{r+1}}\cdots t_s^{a_s})[-a_r]
\stackrel{t_r^{a_r}}{\longrightarrow} \\ 
& & \ \ \ \ \ \ \ \ \ \ \ \ \ \ \ \ \ \ \ \ \ \ S/(L,t^a)\longrightarrow
S/(t_2^{d_2},\ldots,t_{r-1}^{d_{r-1}},t_r^{a_r},t_{r+1}^{d_{r+1}},\ldots,t_s^{d_s})\longrightarrow
0.\nonumber
\end{eqnarray}

Notice that the ring on the right is a complete intersection and the
ring on the left is isomorphic to the tensor product
\begin{equation}\label{dec19-15}
K[t_2,\ldots,t_r]/(t_2^{d_2},\ldots,t_{r-1}^{d_{r-1}},t_r^{d_r-a_r})
\otimes_K K[t_1,t_{r+1},\ldots, t_s]/(t_{r+1}^{d_{r+1}},\ldots,t_s^{d_s},
t_{r+1}^{a_{r+1}}\cdots t_s^{a_s}).
\end{equation}

Hence, taking Hilbert series in Eq.~(\ref{dec18-15}), and applying
\cite[Corollary~3.3]{Sta1}, the theorem of
Hilbert--Serre \cite[p.~58]{Sta1}, and
Proposition~\ref{inter-tensoraa}, we can write the Hilbert series of
$S/(L,t^a)$ as 
\begin{eqnarray*}
& &F(S/(L,t^a),x)=\frac{x^{a_r}(1-x^{d_2})\cdots(1-x^{d_{r-1}})(1-x^{d_r-a_r})}{(1-x)^{r-1}}
\frac{g(x)}{(1-x)}+\\
& &\ \ \ \ \ \ \ \ \ \ \ \ \ \ \ \ \ \
\frac{(1-x^{d_2})\cdots(1-x^{d_{r-1}})(1-x^{a_r})(1-x^{d_{r+1}})\cdots
(1-x^{d_s})}{(1-x)^s},
\end{eqnarray*}
where $g(x)/(1-x)$ is the Hilbert series of the second ring in the
tensor product of Eq.~(\ref{dec19-15}) and $g(1)$ is its degree.  By induction hypothesis 
$$g(1)=d_{r+1}\cdots d_s-(d_{r+1}-a_{r+1})\cdots (d_s-a_s).$$ 

Therefore, writing  $F(S/(L,t^a),x)=h(x)/(1-x)$ with $h(x)\in\mathbb{Z}[x]$ and
$h(1)>0$, and recalling that $h(1)$ is the degree of $S/(L,t^a)$, we
get
\begin{eqnarray*}
\deg(S/(L,t^a))=h(1)&=& d_2\cdots d_{r-1}(d_r-a_r)g(1)+d_2\cdots
d_{r-1}a_rd_{r+1}\cdots d_s\\
&=& d_2\cdots d_s- (d_2-a_2)\cdots(d_s-a_s).\quad \quad\Box
\end{eqnarray*}

\begin{lemma}\label{dec28-15} {\rm (A)} Let $L$ be the ideal of $S$ generated by 
$t_1^{d_2},\ldots,t_{p-1}^{d_p},t_p^{c_p}t_{p+1}^{c_{p+1}},t_{p+2}^{d_{p+2}},\ldots
,t_s^{d_s}$, 
where $p\geq 1$, $1\leq c_p\leq
c_{p+1}$ and $d_i\geq 1$ for all $i$. 
If $t^a=t_1^{a_1}\cdots t_s^{a_s}$ is not in $L$, 
$d_{p+1}=c_p+c_{p+1}$, and $a_i\geq 1$ for some $i$, then the degree of $S/(L,t^a)$ is equal to 
$$
\begin{array}{l}
{\rm (i)}\ \ \ \ d_2\cdots d_s- (c_{p+1}-a_{p+1})\displaystyle
\prod_{i=1}^{p-1}(d_{i+1}-a_i)\prod_{i=p+2}^s(d_i-a_i)\mbox{ if }
a_p\geq c_p;\\
{\rm (ii.1)}\ d_2\cdots d_s-
(c_p-a_p)\displaystyle
\prod_{i=1}^{p-1}(d_{i+1}-a_i)\prod_{i=p+2}^s(d_i-a_i)\mbox{ if }
a_p<c_p,\, a_{p+1}\geq c_{p+1};\\
{\rm (ii.2)}\ d_2\cdots d_s- (d_{p+1}-a_p-a_{p+1})\displaystyle
\prod_{i=1}^{p-1}(d_{i+1}-a_i)\prod_{i=p+2}^s(d_i-a_i)\mbox{ if }
a_p<c_p,\, a_{p+1}<c_{p+1}.
\end{array}
$$
\begin{itemize}
\item[\rm(B)] Let $I$ be a graded ideal such that 
$L={\rm in}_\prec(I)$. If $0\leq k\leq s-2$, $\ell$ are integers such that 
$d=\sum_{i=2}^{k+1}\left(d_i-1\right)+\ell$ and $1\leq \ell \leq
d_{k+2}-1$, then ${\rm fp}_I(d)\leq (d_{k+2}-\ell)d_{k+3}\cdots d_s$.
\end{itemize}
\end{lemma}

\begin{proof} (A) Case (i): Assume $a_p\geq c_p$. If $a_i=0$ for $i\neq
p$, then $t^a=t_p^{a_p}$, and by the first equality of
Lemma~\ref{dec20-15}, and using \cite[Corollary~3.3]{Sta1}, we get
\begin{eqnarray*}
\deg(S/(L,t^a))&=&\deg(S/(t_1^{d_2},\ldots,t_{p-1}^{d_p},t_p^{a_p},t_{p+2}^{d_{p+2}},\ldots
,t_s^{d_s},t_p^{c_p}t_{p+1}^{c_{p+1}}))\\
&=&\deg(S/(t_1^{d_2},\ldots,t_{p-1}^{d_p},t_p^{c_p},t_{p+2}^{d_{p+2}},\ldots
,t_s^{d_s}))=d_2\cdots d_pc_pd_{p+2}\cdots d_s\\
&=&d_2\cdots d_p(d_{p+1}-c_{p+1})d_{p+2}\cdots d_s=d_2\cdots
d_s-c_{p+1}d_2\cdots d_pd_{p+2}\cdots d_s,
\end{eqnarray*}  
as required. We may now assume that $a_i\geq 1$ for some $i\neq p$. 
As $t^a\notin L$ and $a_p\geq c_p$, one has $a_i<d_{i+1}$ for
$i=1,\ldots,p-1$, $a_{p+1}<c_{p+1}$, and $a_i<d_i$ for
$i=p+2,\ldots,s$. Therefore from the exact
sequence 
\begin{eqnarray*}
& & 0\longrightarrow
S/(t_1^{d_2},\ldots,t_{p-1}^{d_p},t_{p+1}^{c_{p+1}},t_{p+2}^{d_{p+2}},\ldots,t_s^{d_s},
t_1^{a_1}\cdots t_{p-1}^{a_{p-1}}t_p^{a_p-c_p}t_{p+1}^{a_{p+1}}\cdots 
t_s^{a_s})[-c_p]
\stackrel{t_p^{c_p}}{\longrightarrow} \\ 
& & \ \ \ \ \  \ \ \ \ \  \ \ \ \ \ \ \ \ S/(L,t^a)\longrightarrow
S/(t_1^{d_2},\ldots,t_{p-1}^{d_p},t_p^{c_p},t_{p+2}^{d_{p+2}},\ldots,t_s^{d_s})\longrightarrow
0,
\end{eqnarray*}
and using Lemma~\ref{dec20-15} and \cite[Corollary~3.3]{Sta1}, 
the required equality follows. 

Case (ii): Assume $a_p<c_p$. If $a_i=0$ for $i\neq
p$, then $t^a=t_p^{a_p}$ and $0=a_{p+1}<c_{p+1}$. Hence, by
\cite[Corollary~3.3]{Sta1}, we get 
\begin{eqnarray*}
\deg(S/(L,t^a))&=&\deg(S/(t_1^{d_2},\ldots,t_{p-1}^{d_p},t_p^{a_p},t_{p+2}^{d_{p+2}},\ldots
,t_s^{d_s}))\\
&=&d_2\cdots d_pa_pd_{p+2}\cdots d_s\\
&=&d_2\cdots
d_s-(d_{p+1}-a_p)d_2\cdots d_pd_{p+2}\cdots d_s,
\end{eqnarray*}  
as required. We may now assume that $a_i\geq 1$ for some $i\neq p$.
Consider the exact
sequence 
\begin{eqnarray}\label{dec18-15-1}
& & 0\longrightarrow
S/(t_1^{d_2},\ldots,t_{p-1}^{d_p},t_{p+1}^{c_{p+1}},t_{p+2}^{d_{p+2}},\ldots,t_s^{d_s},
t_1^{a_1}\cdots t_{p-1}^{a_{p-1}}t_{p+1}^{a_{p+1}}\cdots 
t_s^{a_s})[-c_p]
\stackrel{t_p^{c_p}}{\longrightarrow}\nonumber \\ 
& & \ \ \ \ \  \ \ \ \ \  \ \ \ \ \ \ \ \ S/(L,t^a)\longrightarrow
S/(t_1^{d_2},\ldots,t_{p-1}^{d_p},t_p^{c_p},t_{p+2}^{d_{p+2}},\ldots,t_s^{d_s},
t_1^{a_1}\cdots t_s^{a_s})\longrightarrow
0.
\end{eqnarray}

Subcase (ii.1): Assume $a_{p+1}\geq c_{p+1}$. As $t^a\notin L$, in
our situation, one has $a_i<d_{i+1}$ for 
$i=1,\ldots,p-1$, $a_p<c_p$, and $a_i<d_i$ for
$i=p+2,\ldots,s$. If $a_i=0$ for $i\neq p+1$, then taking Hilbert
series in Eq.~(\ref{dec18-15-1}), and noticing that the ring on the
right has dimension $0$, we get 
\begin{eqnarray*}
\deg(S/(L,t^a))&=&d_2\cdots d_pc_{p+1}d_{p+2}\cdots d_s\\
&=& d_2\cdots d_s-c_pd_2\cdots d_{p}d_{p+2}\cdots d_s,
\end{eqnarray*}
as required. Thus we may now assume that $a_i\geq 1$ for some $i\neq
p+1$. Taking Hilbert series in Eq.~(\ref{dec18-15-1}), and using \cite[Corollary~3.3]{Sta1}, we obtain
\begin{eqnarray*}
\deg(S/(L,t^a))&=&d_2\cdots d_pc_{p+1}d_{p+2}\cdots d_s+\\
& & \deg(S/(t_1^{d_2},\ldots,t_{p-1}^{d_p},t_p^{c_p},t_{p+2}^{d_{p+2}},\ldots,t_s^{d_s},
t_1^{a_1}\cdots t_s^{a_s})).
\end{eqnarray*}  

Therefore, using Lemma~\ref{dec20-15}, the required equality follows.

Subcase (ii.2): Assume $a_{p+1}<c_{p+1}$. If $a_i=0$ for $i\neq p+1$, taking Hilbert
series in Eq.~(\ref{dec18-15-1}), and noticing that the ring on the
right has dimension $0$, by Lemma~\ref{dec20-15}, we get 
\begin{eqnarray*}
\deg(S/(L,t^a))&=&d_2\cdots d_pc_{p+1}d_{p+2}\cdots d_s-d_2\cdots
d_p(c_{p+1}-a_{p+1})d_{p+2}\cdots d_s\\
&=& d_2\cdots d_pa_{p+1}d_{p+2}\cdots d_s\\
& =& d_2\cdots d_s-(d_{p+1}-a_{p+1})d_2\cdots d_pd_{p+2}\cdots d_s,
\end{eqnarray*}
as required. Thus we may now assume that $a_i\geq 1$ for some $i\neq
p+1$. Taking Hilbert series in Eq.~(\ref{dec18-15-1}), and applying
Lemma~\ref{dec20-15} to the ends of Eq.~(\ref{dec18-15-1}), the
required equality follows. 

(B)  It suffices to
find a monomial $t^b$ in $\mathcal{M}_{\prec,d}$ such that 
\begin{equation}\label{feb1-16}
\deg(S/({\rm in}_\prec(I),t^b))=(d_{k+2}-\ell)d_{k+3}\cdots d_s.
\end{equation}
There are five cases to consider:
$$
t^b=\left\{\begin{array}{ll}
t_1^{d_2-1}\cdots
t_{p-1}^{d_p-1}t_p^{c_p}t_{p+1}^{c_{p+1}-1}t_{p+2}^{d_{p+2}-1}\cdots
t_{k+1}^{d_{k+1}-1}t_{k+2}^\ell&\mbox{if }k\geq p+1,\\
t_1^{d_2-1}\cdots
t_{p-1}^{d_p-1}t_p^{c_p}t_{p+1}^{c_{p+1}-1}t_{p+2}^\ell&\mbox{if }k=p ,\\
t_1^{d_2-1}\cdots
t_{k}^{d_{k+1}-1}t_{k+1}^\ell&\mbox{if } k\leq p-2,\\
t_1^{d_2-1}\cdots
t_{p-1}^{d_p-1}t_p^{c_p}t_{p+1}^{\ell-c_{p}}&\mbox{if } k=p-1\mbox{ and }\ell\geq c_p,\\
t_1^{d_2-1}\cdots
t_{p-1}^{d_p-1}t_p^{\ell}&\mbox{if } k=p-1 \mbox{ and }\ell< c_p.
\end{array}\right.
$$
In each case, by the formulas for the degree of part (A), 
we get the equality of Eq.~(\ref{feb1-16}).
\end{proof}

If $f\in S$, the {\it quotient ideal\/} of $I$ with
respect to $f$ is given by $(I\colon f)=\{h\in S\vert\, hf\in I\}$.

\begin{lemma}{\cite[Lemma~4.1]{hilbert-min-dis}}\label{degree-initial-footprint}
Let $I\subset S$ be an unmixed graded ideal and let $\prec$ be 
a monomial order. If $f\in S$ is homogeneous and $(I\colon f)\neq I$, then
$$
\deg(S/(I,f))\leq\deg(S/({\rm
in}_\prec(I),{\rm in}_\prec(f)))\leq\deg(S/I),
$$
and $\deg(S/(I,f))<\deg(S/I)$ if $I$ is an unmixed radical ideal and $f\notin I$.
\end{lemma}

\begin{remark}\rm  
Let $I\subset S$ be an unmixed graded ideal of
dimension $1$. If $f\in S_d$, then
$(I\colon f)=I$ if and only if $\dim(S/(I,f))=0$. In this case 
 $\deg(S/(I,f))$ could be greater than $\deg(S/I)$.
\end{remark}

\begin{lemma}{\cite[Lemma~3.2]
{hilbert-min-dis}}\label{degree-formula-for-the-number-of-zeros-proj}
Let $\mathbb{X}$ be a finite subset of 
$\mathbb{P}^{s-1}$ over a field $K$ and let $I(\mathbb{X})\subset S$ be its
graded vanishing ideal. If 
$0\neq f\in S$ is homogeneous, then the number of zeros of $f$ in
$\mathbb{X}$ is given by 
$$
|V_{\mathbb{X}}(f)|=\left\{
\begin{array}{cl}
\deg S/(I(\mathbb{X}),f)&\mbox{if }(I(\mathbb{X})\colon f)\neq
I(\mathbb{X}),\\ 
0&\mbox{if }(I(\mathbb{X})\colon f)=I(\mathbb{X}).
\end{array}
\right.
$$
\end{lemma}

\begin{corollary}\label{bounds-for-deg-init-ci-case-1} 
Let $I=I(\mathbb{X})$ be the vanishing ideal of a finite set
$\mathbb{X}$ of
projective points, let $f\in\mathcal{F}_{\prec,d}$, and 
${\rm in}_\prec(f)=t_1^{a_1}\cdots t_s^{a_s}$. If ${\rm
in}_\prec(I)$ 
is generated by 
$t_2^{d_2},\ldots ,t_s^{d_s}$, then there is 
$r\geq 2$ such that $a_r\geq 1$, $a_i\leq d_i-1$ for $i\geq r$, 
$a_i=0$ if $2\leq i<r$, and 
$$
|V_{\mathbb{X}}(f)|\leq d_2\cdots d_s- (d_2-a_2)\cdots(d_s-a_s). 
$$
\end{corollary}

\demo 
As $f$ is a zero-divisor of $S/I$, by
Lemma~\ref{regular-elt-in}, $t^a={\rm in}_\prec(f)$ is a zero-divisor
of $S/{\rm in}_\prec(I)$. Hence, there is 
$r\geq 2$ such that $a_r\geq 1$ and $a_i=0$ if $2\leq i<r$. Using 
that $t^a$ is a standard monomial of $S/I$, we get that $a_i\leq d_i-1$ for
$i\geq r$. Therefore, using
Lemma~\ref{degree-formula-for-the-number-of-zeros-proj}  
together with Lemmas~\ref{dec20-15} and \ref{degree-initial-footprint}, we get  
\begin{eqnarray*}
|V_{\mathbb{X}}(f)|&=&\deg(S/(I(\mathbb{X}),f))
\leq\deg(S/({\rm
in}_\prec(I(\mathbb{X})),{\rm in}_\prec(f)))\\
&=&d_2\cdots d_s- (d_2-a_2)\cdots(d_s-a_s).\ \ \ \ \ \ \Box 
\end{eqnarray*}

\begin{lemma}\label{delta-indep-order}
Let $I$ be a graded ideal and let $\prec$ be a monomial order. Then
the minimum distance function $\delta_I$ is independent of $\prec$.
\end{lemma}

\begin{proof} Fix a positive integer $d$. Let $\mathcal{F}_d$ be the set of 
all homogeneous zero-divisors of $S/I$ not in $I$ of degree $d$ 
and let $f$ be an element of $\mathcal{F}_d$. Pick a Gr\"obner basis
$g_1,\ldots,g_r$ of $I$. Then, by the division algorithm
\cite[Theorem~3, p.~63]{CLO}, we can write 
$f=\sum_{i=1}^ra_ig_i+h$, 
where $h$ is a homogeneous standard polynomial of $S/I$ of degree $d$. 
Since $(I\colon f)=(I\colon h)$, we get that $h$ is in
$\mathcal{F}_{\prec,d}$. Hence, as $(I,f)=(I,h)$, we get the
equalities:
\begin{eqnarray*}
\delta_I(d)&=&
\deg(S/I)-\max\{\deg(S/(I,f))\vert\, f\in\mathcal{F}_{\prec, d}\}\\
&=&\deg(S/I)-\max\{\deg(S/(I,f))\vert\, f\in\mathcal{F}_d\},
\end{eqnarray*}
that is, $\delta_I(d)$ does not depend on $\mathcal{F}_{\prec,d}$.
\end{proof}

\begin{lemma}\label{dec30-15} Let $I$ be an unmixed graded ideal and
$\prec$ a monomial order. The following hold. 
\begin{itemize}
\item[(a)] $\delta_I(d)\geq {\rm fp}_I(d)$ and 
$\delta_I(d)\geq 0$ for $d\geq 1$.  
\item[(b)] ${\rm fp}_I(d)\geq 0$ if ${\rm in}_\prec(I)$ is unmixed. 
\item[(c)] $\delta_I(d)\geq 1$ if $I$ is radical.
\end{itemize}
\end{lemma}

\begin{proof} If $\mathcal{F}_{\prec,d}=\emptyset$, then
clearly $\delta_I(d)=\deg(S/I)\geq 1$, $\delta_I(d)\geq {\rm fp}_I(d)$, and if
${\rm in}_\prec(I)$ is unmixed, then ${\rm fp}_I(d)\geq 0$ (this
follows from Lemma~\ref{degree-initial-footprint}). Thus, (a), (b),
and (c) hold. Now assume that $\mathcal{F}_{\prec,d}\neq\emptyset$. Pick
a standard polynomial $f\in S_d$ such that 
$(I\colon f)\neq I$ and 
$$
\delta_I(d)=\deg(S/I)-\deg(S/(I,f)).
$$

As $I$ is unmixed, by
Lemma~\ref{degree-initial-footprint}, $\deg(S/(I,f))\leq 
\deg(S/({\rm in}_\prec(I),{\rm in}_\prec(f)))$. On the other hand, by
Lemma~\ref{regular-elt-in}, ${\rm in}_\prec(f)$ is a zero-divisor of
$S/{\rm in}_\prec(I)$. Hence $\delta_I(d)\geq {\rm fp}_I(d)$. Using the second
inequality of Lemma~\ref{degree-initial-footprint} it follows that 
$\delta_I(d)\geq 0$, ${\rm fp}_I(d)\geq 0$ if ${\rm in}_\prec(I)$ is
unmixed, and $\delta_I(d)\geq 1$ if $I$ is radical.
\end{proof}

\begin{proposition}\label{geil-carvalho-monomial} If $I$ is an unmixed monomial ideal and $\prec$ is
any monomial order, then $\delta_I(d)={\rm fp}_I(d)$ for $d\geq 1$,
that is, $I$ is a Geil--Carvalho ideal.
\end{proposition}

\begin{proof} The inequality $\delta_I(d)\geq {\rm fp}_I(d)$ follows
from Lemma~\ref{dec30-15}. To show the reverse inequality 
notice that $\mathcal{M}_{\prec, d}\subset
\mathcal{F}_{\prec, d}$ because one has $I={\rm in}_\prec(I)$. Also
notice that $\mathcal{M}_{\prec, d}=\emptyset$ if and only if 
$\mathcal{F}_{\prec, d}=\emptyset$, this follows from
Lemma~\ref{regular-elt-in}. Therefore one has 
${\rm fp}_I(d)\geq \delta_I(d)$.
\end{proof}

\begin{proposition}\label{jan2-16} 
Let $I\subset S$ be a graded ideal and let $\prec$ be a
monomial order. Suppose that ${\rm in}_\prec(I)$ is a complete intersection of
height $s-1$ generated by 
$t^{\alpha_2},\ldots,t^{\alpha_s}$, with $d_i=\deg(t^{\alpha_i})$ and 
$d_i\geq 1$ for all $i$. The following hold.
\begin{itemize}
\item[(a)] \cite[Example~1.5.1]{Migliore} $I$ is a complete
intersection and $\dim(S/I)=1$.
\item[(b)] {\rm(}\cite[Example~1.5.1]{Migliore},
\cite[Lemma~3.5]{Chardin}{\rm)} $\deg(S/I)=d_2\cdots d_s$ and 
${\rm reg}(S/I)=\sum_{i=2}^s(d_i-1)$. 
\item[(c)] $1\leq {\rm fp}_I(d)\leq \delta_I(d)$ for $d\geq 1$.
\end{itemize}
\end{proposition}
\begin{proof} (a): The rings $S/I$ and $S/{\rm init}_\prec(I)$ have the same
dimension. Thus $\dim(S/I)=1$. As $\prec$ is a graded order, there are 
$f_2,\ldots,f_s$ homogeneous polynomials in $I$ with 
${\rm in}_\prec(f_i)=t^{\alpha_i}$ for $i\geq 2$. Since 
${\rm in}_\prec(I)=({\rm in}_\prec(f_2),\ldots,{\rm in}_\prec(f_s))$,
the polynomials $f_2,\ldots,f_s$ form a Gr\"obner basis of $I$, and in
particular they generate $I$. Hence $I$ is a graded ideal of height $s-1$
generated by $s-1$ polynomials, that is, $I$ is a complete intersection.

(b): Since $I$ is a complete intersection generated by the $f_i$'s,
then the degree and regularity of $S/I$ are $\deg(f_2)\cdots\deg(f_s)$
and $\sum_{i=2}^s(\deg(f_i)-1)$, respectively. This follows from the 
formula for the Hilbert series of a complete intersection given in 
\cite[Corollary~3.3]{Sta1}. 

(c) The ideal $I$ is unmixed because, by part (a), $I$ is a complete
intersection; in particular Cohen--Macaulay and unmixed. 
Hence the inequality $\delta_I(d)\geq {\rm fp}_I(d)$
follows from Lemma~\ref{dec30-15}. Let $t^a$ be a standard monomial
of $S/I$ of degree $d$ such that $({\rm in}_\prec(I)\colon t^a)\neq {\rm
in}_\prec(I)$, that is, $t^a$ is in $\mathcal{M}_{\prec, d}$. 
Using Lemma~\ref{ci-monomial-dim1}, and the formulas for $\deg(S/({\rm
in}_\prec(I),t^a))$ given in Lemma~\ref{dec20-15} and
Lemma~\ref{dec28-15}, we obtain that 
$\deg(S/({\rm in}_\prec(I),t^a))<\deg(S/I)$. Thus ${\rm fp}_I(d)\geq 1$.
\end{proof}

\begin{proposition}{\cite[Proposition~5.7]{hilbert-min-dis}}\label{aug-28-15}
Let $1\leq e_1\leq\cdots\leq e_m$ and $0\leq b_i\leq e_i-1$
for $i=1,\ldots,m$ be integers. If $b_0\geq 1$, then 
\begin{equation}\label{aug-27-15-2}
\prod_{i=1}^m(e_i-b_i)\geq
\left(\sum_{i=1}^{k+1}(e_i-b_i)-(k-1)-b_0-\sum_{i=k+2}^m b_i\right)e_{k+2}\cdots e_m
\end{equation}
for $k=0,\ldots,m-1$, where $e_{k+2}\cdots e_m=1$ and
$\sum_{i=k+2}^mb_i=0$ if $k=m-1$.
\end{proposition}

We come to the main result of this paper.

\begin{theorem}\label{footprint-ci}
Let $I\subset S$ be a graded ideal and let $\prec$ be a 
graded monomial order. If the initial ideal 
${\rm in}_\prec(I)$ is a complete intersection of
height $s-1$ generated by 
$t^{\alpha_2},\ldots,t^{\alpha_s}$, with $d_i=\deg(t^{\alpha_i})$ and 
$1\leq d_i\leq d_{i+1}$ for $i\geq 2$, then $\delta_I(d)\geq {\rm
fp}_I(d)\geq 1$ for $d\geq 1$ and   
$$
{\rm fp}_I(d)=\left\{\begin{array}{ll}(d_{k+2}-\ell)d_{k+3}\cdots d_s
&\mbox{if }\ 
d\leq \sum\limits_{i=2}^{s}\left(d_i-1\right)-1,\\ 
1 &\mbox{if\ }\ d\geq \sum\limits_{i=2}^{s}\left(d_i-1\right),
\end{array}\right.
$$
where $0\leq k\leq s-2$ and $\ell$ are integers such that 
$d=\sum_{i=2}^{k+1}\left(d_i-1\right)+\ell$ and $1\leq \ell \leq
d_{k+2}-1$. 
\end{theorem}

\begin{proof} Let $t^a$ be any standard monomial of $S/I$ of degree $d$ which is a zero-divisor 
of $S/{\rm in}_\prec(I)$, that is, $t^a$ is in
$\mathcal{M}_{\prec,d}$. Thus $d=\sum_{i=1}^sa_i$, where
$a=(a_1,\ldots,a_s)$. We set 
$r=\sum_{i=2}^s(d_i-1)$. 
If we substitute $-\ell=\sum_{i=2}^{k+1}(d_i-1)-\sum_{i=1}^sa_i$ in 
the expression $(d_{k+2}-\ell)d_{k+3}\cdots d_s$, it follows that for $d<r$ the inequality 
$$
{\rm fp}_I(d)\geq (d_{k+2}-\ell)d_{k+3}\cdots d_s
$$
is equivalent to show that 
\begin{equation}\label{dec28-15-1}
\deg(S/I)- \deg(S/({\rm
in}_\prec(I),t^a))\geq\left(\sum_{i=2}^{k+2}(d_i-a_i)-k-a_1-\sum_{i=k+3}^sa_i\right)d_{k+3}\cdots d_s
\end{equation}
for any $t^a$ in
$\mathcal{M}_{\prec,d}$, where by convention $\sum_{i=k+3}^sa_i=0$ and $d_{k+3}\cdots
d_s=1$ if $k=s-2$. Recall that, by Proposition~\ref{jan2-16}, one has
that ${\rm fp}_I(d)\geq 1$ for $d\geq 1$. By Lemma~\ref{ci-monomial-dim1}, and by permuting
variables and changing $I$, $\prec$, and $t^a$ accordingly, one has the
following two cases to consider.

Case (i): Assume that ${\rm
in}_\prec(I)=(t_2^{d_2},\ldots,t_s^{d_s})$ with $1\leq d_i\leq
d_{i+1}$ for $i\geq 2$. Then, as $t^a$ is in $\mathcal{M}_{\prec,
d}$, 
we can write $t^a=t_1^{a_1}\cdots t_r^{a_r}\cdots
t_s^{a_s}$, $r\geq 2$, $a_r\geq 1$, $a_i=0$ if $2\leq i<r$, and
$a_i\leq d_i-1$ for $i\geq r$. By Lemma~\ref{dec20-15} we get
\begin{equation}\label{dec20-15-1}
\deg(S/({\rm in}_\prec(I),t^a)) = d_2\cdots d_s-
(d_2-a_2)\cdots(d_s-a_s)
\end{equation}
for any $t^a$ in $\mathcal{M}_{\prec,d}$. If $d\geq r$, setting
$t^c=t_1^{d-r}t_2^{d_2-1}\cdots t_s^{d_s-1}$, one has
$t^c\in\mathcal{M}_{\prec,d}$. Then, using Eq.~(\ref{dec20-15-1}), it follows that 
  $\deg(S/({\rm in}_\prec(I),t^c))=d_2\cdots d_s-1$. Thus 
${\rm  fp}_I(d)\leq 1$ and equality ${\rm fp}_I(d)=1$ holds. We
may now assume $d\leq r-1$. Setting $t^b=t_2^{d_2-1}\cdots
t_{k+1}^{d_{k+1}-1}t_{k+2}^\ell$, one has
$t^b\in\mathcal{M}_{\prec,d}$. Then, using Eq.~(\ref{dec20-15-1}), we
get 
$$
\deg(S/({\rm in}_\prec(I),t^b))=d_2\cdots
d_s-(d_{k+2}-\ell)d_{k+3}\cdots d_s.
$$ 

Hence ${\rm fp}_I(d)\leq (d_{k+2}-\ell)d_{k+3}\cdots d_s$. 
Next we show the reverse inequality by showing that
the inequality of Eq.~(\ref{dec28-15-1}) holds for any
$t^a\in\mathcal{M}_{\prec,d}$. By Eq.~(\ref{dec20-15-1}) it suffices to 
show that the following equivalent inequality holds
$$
(d_2-a_2)\cdots(d_s-a_s)\geq\left(\sum_{i=2}^{k+2}(d_i-a_i)-k-a_1
-\sum_{i=k+3}^sa_i\right)d_{k+3}\cdots d_s
$$
for any $a=(a_1,\ldots,a_s)$ such that $t^a\in\mathcal{M}_{\prec,d}$.
This inequality follows from 
Proposition~\ref{aug-28-15} by making $m=s-1$, $e_i=d_{i+1}$, $b_i=a_{i+1}$ 
for $i=1,\ldots,s-1$ and $b_0=1+a_1$.

Case (ii): Assume that 
${\rm in}_\prec(I)=(t_1^{d_2},\ldots,t_{p-1}^{d_p},t_p^{c_p}t_{p+1}^{c_{p+1}},t_{p+2}^{d_{p+2}},\ldots
,t_s^{d_s})$ for some $p\geq 1$ such that $1\leq c_p\leq
c_{p+1}$ and $1\leq d_i\leq d_{i+1}$ for all $i$, where
$d_{p+1}=c_p+c_{p+1}$. 

If $d\geq r$, setting
$t^c=t_1^{d_2-1}\cdots
t_{p-1}^{d_p-1}t_p^{d-r+c_p}t_{p+1}^{c_{p+1}-1}t_{p+2}^{d_{p+2}-1}\cdots t_s^{d_s-1}$, we get that
$t^c\in\mathcal{M}_{\prec,d}$. Then, using the first formula of
Lemma~\ref{dec28-15}, it follows that $\deg(S/({\rm
in}_\prec(I),t^c))=d_2\cdots d_s-1$. Thus ${\rm fp}_I(d)\leq 1$ and
the equality ${\rm fp}_I(d)=1$ holds. 

We may now assume $d\leq r-1$. The inequality 
${\rm fp}_I(d)\leq (d_{k+2}-\ell)d_{k+3}\cdots d_s$ follows from
Lemma~\ref{dec28-15}(B). To show that 
${\rm fp}_I(d)\geq (d_{k+2}-\ell)d_{k+3}\cdots d_s$ we need only show that
the inequality of Eq.~(\ref{dec28-15-1}) holds for any $t^a$ in
$\mathcal{M}_{\prec, d}$. Take $t^a$ in 
$\mathcal{M}_{\prec, d}$. Then we can write 
$t^a=t_1^{a_1}\cdots t_s^{a_s}$ with $a_i<d_{i+1}$ for $i<p$ and 
$a_i<d_i$ for $i>p+1$. There are three subcases to consider.

Subcase (ii.1): Assume $a_p\geq c_p$. Then $c_{p+1}>a_{p+1}$ because
$t^a$ is a standard monomial of $S/I$, and by Lemma~\ref{dec28-15} we get
\begin{equation*}
\deg(S/({\rm in}_\prec(I),t^a))=
d_2\cdots d_s- (c_{p+1}-a_{p+1})\displaystyle
\prod_{i=1}^{p-1}(d_{i+1}-a_i)\prod_{i=p+2}^s(d_i-a_i).
\end{equation*}
Therefore the inequality of Eq.~(\ref{dec28-15-1}) is equivalent to
\begin{eqnarray*}
& &(c_{p+1}-a_{p+1})\displaystyle
\prod_{i=1}^{p-1}(d_{i+1}-a_i)\prod_{i=p+2}^s(d_i-a_i)\\ 
& & \ \ \ \ \ \ \ \ \ \ \ \ \ \ 
\geq\left(\sum_{i=2}^{k+2}(d_i-a_i)-k-a_1-\sum_{i=k+3}^sa_i\right)d_{k+3}\cdots
d_s,
\end{eqnarray*}
and this inequality follows at once from
Proposition~\ref{aug-28-15} by making $m=s-1$, $e_i=d_{i+1}$ for
$i=1,\ldots,m$, $b_i=a_i$ for $1\leq i\leq p-1$, $b_p=a_{p+1}+c_p$,  
$b_i=a_{i+1}$ for $p<i\leq m$, and $b_0=a_p-c_p+1$. Notice 
that $\sum_{i=0}^mb_i=1+\sum_{i=1}^sa_i$. 

Subcase (ii.2): Assume $a_p<c_p$, $a_{p+1}\geq c_{p+1}$. By Lemma~\ref{dec28-15} we get
\begin{equation*}
\deg(S/({\rm in}_\prec(I),t^a))=
d_2\cdots d_s- (c_p-a_p)\displaystyle
\prod_{i=1}^{p-1}(d_{i+1}-a_i)\prod_{i=p+2}^s(d_i-a_i).
\end{equation*}
Therefore the inequality of Eq.~(\ref{dec28-15-1}) is equivalent to 
\begin{eqnarray*}
& &(c_p-a_p)\displaystyle
\prod_{i=1}^{p-1}(d_{i+1}-a_i)\prod_{i=p+2}^s(d_i-a_i)\\ 
& & \ \ \ \ \ \ \ \ \ \ \ \ \ \ 
\geq\left(\sum_{i=2}^{k+2}(d_i-a_i)-k-a_1-\sum_{i=k+3}^sa_i\right)d_{k+3}\cdots
d_s,
\end{eqnarray*}
and this inequality follows from
Proposition~\ref{aug-28-15} by making $m=s-1$, $e_i=d_{i+1}$ for
$i=1,\ldots,m$, $b_i=a_i$ for $1\leq i\leq p-1$, $b_p=c_{p+1}+a_p$,  
$b_i=a_{i+1}$ for $p<i\leq m$, and $b_0=a_{p+1}-c_{p+1}+1$. Notice 
that $\sum_{i=0}^mb_i=1+\sum_{i=1}^sa_i$. 

Subcase (ii.3): Assume $a_p<c_p$, $a_{p+1}\leq c_{p+1}-1$. 
By Lemma~\ref{dec28-15} we get
\begin{equation*}
\deg(S/({\rm in}_\prec(I),t^a))=
d_2\cdots d_s- (d_{p+1}-a_p-a_{p+1})\displaystyle
\prod_{i=1}^{p-1}(d_{i+1}-a_i)\prod_{i=p+2}^s(d_i-a_i).
\end{equation*}
Therefore the inequality of Eq.~(\ref{dec28-15-1}) is equivalent to 
\begin{eqnarray*}
& &(d_{p+1}-a_p-a_{p+1})\displaystyle
\prod_{i=1}^{p-1}(d_{i+1}-a_i)\prod_{i=p+2}^s(d_i-a_i)\\ 
& & \ \ \ \ \ \ \ \ \ \ \ \ \ \ 
\geq\left(\sum_{i=2}^{k+2}(d_i-a_i)-k-a_1-\sum_{i=k+3}^sa_i\right)d_{k+3}\cdots
d_s,
\end{eqnarray*}
and this inequality follows from
Proposition~\ref{aug-28-15} by making $m=s-1$, $e_i=d_{i+1}$ for
$i=1,\ldots,m$, $b_i=a_i$ for $1\leq i\leq p-1$, $b_p=a_{p}+a_{p+1}$,  
$b_i=a_{i+1}$ for $p<i\leq m$, and $b_0=1$. Notice 
that in this case $\sum_{i=0}^mb_i=1+\sum_{i=1}^sa_i$. 
\end{proof} 

\section{Applications and examples}
This section is devoted to give some applications and examples of our
main result. As the two most important applications to algebraic
coding theory, we recover the formula
for the minimum distance of an affine cartesian code 
\cite{geil-thomsen,cartesian-codes}, and the fact that the
homogenization of the corresponding vanishing ideal is a 
Geil--Carvalho ideal \cite{carvalho}.

We begin with a basic application for
complete intersections in $\mathbb{P}^1$.

\begin{corollary}\label{p1} If $\mathbb{X}$ is a finite subset of $\mathbb{P}^1$
and $I(\mathbb{X})$ is a complete intersection,
then 
$$
\delta_{I(\mathbb{X})}(d)={\rm fp}_{I(\mathbb{X})}(d)=
\left\{\begin{array}{cl} |\mathbb{X}|-d &\mbox{if }1\leq d\leq
|\mathbb{X}|-2,\\
1 &\mbox{if }d\geq |\mathbb{X}|-1.
\end{array}\right.
$$ 
\end{corollary} 

\begin{proof} Let $f$ be the generator of $I(\mathbb{X})$. In this
case $d_2=\deg(f)=|\mathbb{X}|$ and 
${\rm reg}(S/I(\mathbb{X}))=|\mathbb{X}|-1$. By Proposition~\ref{jan4-15} and
Theorem~\ref{footprint-ci} one has 
$$
\delta_\mathbb{X}(d)=\delta_{I(\mathbb{X})}(d)\geq
{\rm fp}_{I(\mathbb{X})}(d)=|\mathbb{X}|-d\ \mbox{ for }\ 1\leq d\leq
|\mathbb{X}|-2,
$$
and $\delta_\mathbb{X}(d)=1$ for $d\geq |\mathbb{X}|-1$. Assume that
$1\leq d\leq |\mathbb{X}|-2$. Pick $[P_1],\ldots,[P_d]$ points in
$\mathbb{P}^1$. By Lemma~\ref{primdec-ix-a}, the vanishing ideal
$I_{[P_i]}$ of $[P_i]$ is a principal ideal generated by a linear form
$h_i$. Notice that $V_\mathbb{X}(h_i)$, the zero-set of $h_i$ in
$\mathbb{X}$, is equal to $\{[P_i]\}$. Setting $h=h_1\cdots h_d$, we
get a homogeneous polynomial of degree $d$ with exactly $d$ zeros.
Thus $\delta_\mathbb{X}(d)\leq|\mathbb{X}|-d$.
\end{proof}

As another application we get the following uniform upper bound for 
the number of zeroes of all polynomials $f\in S_d$ that do not 
vanish at all points of $\mathbb{X}$.

\begin{corollary}\label{uniform-bound-for-zeros}
Let $\mathbb{X}$ be a finite subset of $\mathbb{P}^{s-1}$, let 
$I(\mathbb{X})$ be its vanishing
ideal, and let $\prec$ be a monomial order. If\, the initial ideal 
${\rm in}_\prec(I(\mathbb{X}))$ is a complete
intersection generated by 
$t^{\alpha_2},\ldots,t^{\alpha_s}$, with $d_i=\deg(t^{\alpha_i})$ and 
$1\leq d_i\leq d_{i+1}$ for $i\geq 2$, then
\begin{equation}\label{eq-unif-b}
|V_\mathbb{X}(f)|\leq \deg(S/I(\mathbb{X}))-\left(d_{k+2}-\ell\right)d_{k+3}\cdots d_s,
\end{equation}
for any $f\in S_d$ that does not vanish at all point of $\mathbb{X}$,
where $0\leq k\leq s-2$ and $\ell$ are integers such that 
$d=\sum_{i=2}^{k+1}\left(d_i-1\right)+\ell$ and $1\leq \ell \leq
d_{k+2}-1$. 
\end{corollary}

\begin{proof} It follows from Corollary~\ref{min-dis-vi},
Eq.~(\ref{min-dis-vi-coro}), and
Theorem~\ref{footprint-ci}. 
\end{proof}

We leave as an open question whether this uniform bound is optimal, that is, 
whether the equality is attained for some polynomial $f$. Another
open question is whether Corollary~\ref{uniform-bound-for-zeros} is
true if we only assume that $I(\mathbb{X})$ is a complete
intersection. This is related to the following 
conjecture of Toh\v{a}neanu and Van Tuyl. 

\begin{conjecture}\cite[Conjecture~4.9]{tohaneanu-vantuyl} Let $\mathbb{X}$ be a
finite set of points in $\mathbb{P}^{s-1}$. If $I(\mathbb{X})$ is a
complete intersection generated by $f_1,\ldots,f_{s-1}$, with 
$e_i=\deg(f_i)$ for $i=1,\ldots,s-1$, and $2\leq e_i\leq e_{i+1}$ for
all $i$, then $\delta_\mathbb{X}(1)\geq (e_1-1)e_2\cdots e_{s-1}$.
\end{conjecture}

Notice that by Corollary~\ref{uniform-bound-for-zeros} this conjecture is true if 
${\rm in}_\prec(I(\mathbb{X}))$ is a complete intersection, and it is
also true for $s=2$ (see Corollary~\ref{p1}). 

\paragraph{\bf Affine cartesian codes and coverings by hyperplanes}
Given a collection of finite subsets $A_2,\ldots,A_s$ of a field $K$, we denote the
image of 
$$X^*=A_2\times\cdots\times A_s$$
under the map  
$\mathbb{A}^{s-1}\mapsto \mathbb{P}^{s-1}$, $x\mapsto [(1,x)]$, by 
$\mathbb{X}=[1\times A_2\times\cdots \times A_s]$. The affine 
Reed-Muller-type code $C_{\mathbb{X}^*}(d)$ of degree $d$ is called an {\it affine cartesian
code\/} \cite{cartesian-codes}. The basic parameters of the projective
Reed-Muller-type code $C_\mathbb{X}(d)$ are equal to those of
$C_{\mathbb{X}^*}(d)$ \cite{affine-codes}.

A formula for the minimum distance of an affine cartesian code is
given in \cite[Theorem~3.8]{cartesian-codes} and in
\cite[Proposition~5]{geil-thomsen}. A short and elegant
 proof of this formula was given by Carvalho
in \cite[Proposition~2.3]{carvalho}, where he shows that the 
best way to study the minimum distance of an affine cartesian code
is by using the footprint. 
As an application of Theorem~\ref{footprint-ci} we also recover the
formula for the minimum distance of an affine cartesian code  by
examining the underlying vanishing ideal and show that 
this ideal is Geil--Carvalho.

\begin{corollary}{\cite{carvalho,geil-thomsen,cartesian-codes}}\label{lopez-renteria-vila}  
Let $K$ be a field and let $C_\mathbb{X}(d)$ be the projective Reed-Muller type
code of degree $d$ on the 
finite set $\mathbb{X}=[1\times A_2\times\cdots\times
A_s]\subset\mathbb{P}^{s-1}$. If $1\leq
d_i\leq d_{i+1}$ for $i\geq 2$, with $d_i=|A_i|$, and $d\geq 1$, then
the minimum distance of 
$C_\mathbb{X}(d)$ is given by 
$$
\delta_\mathbb{X}(d)=\left\{\hspace{-1mm}
\begin{array}{ll}\left(d_{k+2}-\ell\right)d_{k+3}\cdots d_s&\mbox{ if }
d\leq \sum\limits_{i=2}^{s}\left(d_i-1\right)-1,\\
\qquad \qquad 1&\mbox{ if } d\geq \sum\limits_{i=2}^{s}\left(d_i-1\right),
\end{array}
\right.
$$
and $I(\mathbb{X})$ is Geil--Carvalho, that is,
$\delta_{I(\mathbb{X})}(d)={\rm fp}_{I(\mathbb{X})}(d)$ for $d\geq 1$,
where $k\geq 0$, $\ell$ are the unique integers such that 
$d=\sum_{i=2}^{k+1}\left(d_i-1\right)+\ell$ and $1\leq \ell \leq
d_{k+2}-1$. 
\end{corollary}

\begin{proof} Let $\succ$ be the reverse lexicographical order on
$S$ with $t_2\succ\cdots\succ t_s\succ t_1$. Setting 
$f_i=\prod_{\gamma\in A_i}(t_i-\gamma t_1)$ for $i=2,\ldots,s$, one
has that $f_2,\ldots,f_s$ is a Gr\"obner basis of $I(\mathbb{X})$
whose initial ideal is generated by $t_2^{d_2},\ldots,t_s^{d_s}$ (see
\cite[Proposition~2.5]{cartesian-codes}). By Theorem~\ref{min-dis-vi}
one has the equality $\delta_\mathbb{X}(d)=\delta_{I(\mathbb{X})}(d)$
for $d\geq 1$. 
Thus the inequality ``$\geq$''
follows at once from Theorem~\ref{footprint-ci}. 
This is 
the difficult part of the proof. The rest of the argument reduces to finding 
an appropriate polynomial $f$ where equality occurs, and to using that 
the minimum distance $\delta_\mathbb{X}(d)$ is $1$ for $d$ greater than or equal to 
${\rm reg}(S/I(\mathbb{X}))$. 

We set $r=\sum_{i=2}^s(d_i-1)$. By Propositions~\ref{jan4-15} and
\ref{jan2-16},  
the regularity and the degree of $S/I(\mathbb{X})$ are $r$ and
$|\mathbb{X}|=d_2\cdots d_s$, respectively.  
Assume that
$d<r$. To show the inequality ``$\leq$'' notice that 
there is a polynomial $f\in S_d$ which is a product of linear forms 
such that
$|V_\mathbb{X}(f)|$, the number of zeros of $f$ in $\mathbb{X}$, is
equal to 
$$d_2\cdots d_s-(d_{k+2}-\ell)d_{k+3}\cdots d_s,$$
see \cite[p.~15]{cartesian-codes}. Hence 
$\delta_\mathbb{X}(d)$ is less than or  
equal to 
$(d_{k+2}-\ell)d_{k+3}\cdots d_s$. Thus the required equality
holds. If $d\geq r$, by Proposition~\ref{jan4-15},
$\delta_\mathbb{X}(d)=1$ for $d\geq r$. 
Therefore, by Theorem~\ref{footprint-ci}, $I(\mathbb{X})$ is
Geil--Carvalho.
\end{proof}

The next result is an extension of a result of Alon and F\"uredi 
\cite[Theorem~1]{alon-furedi} that can be applied to any finite subset 
of a projective space whose vanishing ideal has a complete intersection
initial ideal relative to a graded monomial order.

\begin{corollary}\label{jan1-16} Let $\mathbb{X}$ be a finite subset
of a projective space $\mathbb{P}^{s-1}$ and let $\prec$ be a 
monomial order such that ${\rm in}_\prec(I(\mathbb{X}))$ is a
complete intersection 
generated by 
$t^{\alpha_2},\ldots,t^{\alpha_s}$, with $d_i=\deg(t^{\alpha_i})$ 
and $1\leq d_i\leq d_{i+1}$ for all $i$. If the hyperplanes
$H_1,\ldots,H_d$ in $\mathbb{P}^{s-1}$ avoid a point $[P]$ in
$\mathbb{X}$ but otherwise cover all the other $|\mathbb{X}|-1$ points
of $\mathbb{X}$, then $d\geq{\rm reg}(S/I(\mathbb{X}))=\sum_{i=2}^s(d_i-1)$.
\end{corollary}

\begin{proof} Let $h_1,\ldots,h_d$ be the linear forms in $S_1$
that define $H_1,\ldots,H_d$, respectively. Assume that
$d<\sum_{i=2}^s(d_i-1)$. Consider the polynomial $h=h_1\cdots
h_d$. Notice that $h\notin I(\mathbb{X})$ because $h(P)\neq 0$, and
$h(Q)=0$ for all $[Q]\in\mathbb{X}$ with $[Q]\neq[P]$. 
By Theorem~\ref{footprint-ci}, $\delta_\mathbb{X}(d)\geq {\rm
fp}_{I(\mathbb{X})}(d)\geq 2$. Hence, $h$ does not vanish in at 
least two points of $\mathbb{X}$, a contradiction. 
\end{proof}

\begin{example}\label{covering-by-hyperplanes-example}
Let $S$ be the polynomial ring $\mathbb{F}_3[t_1,t_4,t_3,t_2]$ 
with the lexicographical order $t_1\prec
t_4\prec
t_3\prec t_2$, and let $I=I(\mathbb{X})$ be the vanishing
ideal of 
\begin{eqnarray*}
\mathbb{X}&=&\{[(1,0,0,0)],\, [(1,1,1,0)],\, [(1,-1,-1,0)],\, [(1,1,0,1)],\, 
\\ & &\ \ \ \ [(1,-1,1,1)],\, [(1,0,-1,1)],\, 
[(1,-1,0,-1)],\,[(1,0,1,-1)],\,[(1,1,-1,-1)]\}.
\end{eqnarray*}

Using the procedure below in {\it Macaulay\/}$2$ \cite{mac2} and
Theorem~\ref{footprint-ci}, we obtain the following information. The ideal
$I(\mathbb{X})$ is generated by $t_2-t_3-t_4$, $t_3^3-t_3t_1^2$,
and $t_4^3-t_4t_1^2$.  
The regularity and the degree of
$S/I(\mathbb{X})$ are $4$ and $9$, respectively, and $I(\mathbb{X})$
is a Geil--Carvalho ideal whose initial ideal is a complete
intersection generated by $t_2,\, t_3^3,\, t_4^3$. The basic parameters of
the Reed-Muller-type code $C_\mathbb{X}(d)$ are shown in the following
table.

\begin{eqnarray*}
\hspace{-11mm}&&\left.
\begin{array}{c|c|c|c|c}
d & 1 & 2 & 3 & 4\\
   \hline
 |\mathbb{X}| & 9 & 9 & 9 &9
 \\ 
   \hline
 H_\mathbb{X}(d)    \    & 3 & 6 &8 & 9
 \\   
   \hline
 \delta_{\mathbb{X}}(d) &6& 3& 2& 1 \\ 
\hline
 {\rm fp}_{I(\mathbb{X})}(d) &6& 3& 2& 1 \\ 
\end{array}
\right.
\end{eqnarray*}

By Corollary~\ref{jan1-16}, if the hyperplanes
$H_1,\ldots,H_d$ in $\mathbb{P}^{3}$ avoid a point $[P]$ in
$\mathbb{X}$ but otherwise cover all the other $|\mathbb{X}|-1$ points
of $\mathbb{X}$, then $d\geq{\rm reg}(S/I(\mathbb{X}))=4$.
\begin{verbatim}
S=ZZ/3[t2,t3,t4,t1,MonomialOrder=>Lex];
I1=ideal(t2,t3,t4),I2=ideal(t4,t3-t1,t2-t1),I3=ideal(t4,t1+t3,t2+t1)
I4=ideal(t4-t1,t4-t2,t3),I5=ideal(t4-t1,t3-t1,t2+t1),I6=ideal(t2,t1-t4,t3+t1)
I7=ideal(t3,t1+t4,t1+t2),I8=ideal(t2,t4+t1,t3-t1),I9=ideal(t1+t4,t3+t1,t2-t1)
I=intersect(I1,I2,I3,I4,I5,I6,I7,I8,I9)
M=coker gens gb I, regularity M, degree M
h=(d)->degree M - max apply(apply(apply(apply (toList
(set(0..q-1))^**(hilbertFunction(d,M))-(set{0})^**(hilbertFunction(d,M)),
toList),x->basis(d,M)*vector x),z->ideal(flatten entries z)),
x-> if not quotient(I,x)==I then degree ideal(I,x) else 0)--this
--gives the minimum distance in degree d
apply(1..3,h)
\end{verbatim}
\end{example}

\begin{example}\label{yuriko-example}
Let $S$ be the polynomial ring $S=\mathbb{F}_3[t_1,t_2,t_3]$ with the 
lexicographical order $t_1\succ t_2\succ t_3$, and let $I=I(\mathbb{X})$ be the vanishing
ideal of 
$$
\mathbb{X}=\{[(1,1,0)],\, [(1,-1,0)],\, [(1,0,1)],\, [(1,0,-1)],[(1,-1,-1)],[(1,1,1)]\}.
$$

As in Example~\ref{covering-by-hyperplanes-example}, using 
{\it Macaulay\/}$2$ \cite{mac2}, we get that 
$I(\mathbb{X})$ is generated by 
$$
t_2^2t_3-t_2t_3^2,\ \, t_1^2-t_2^2+t_2t_3-t_3^2.
$$

The regularity and the degree of
$S/I(\mathbb{X})$ are $3$ and $6$, respectively, $I$
is a Geil--Carvalho ideal, and ${\rm in}_\prec(I)$ is a complete
intersection generated by $t_2^2t_3$ and $t_1^2$. The basic parameters of
the Reed-Muller-type code $C_\mathbb{X}(d)$ are shown in the following
table.

\begin{eqnarray*}
\hspace{-11mm}&&\left.
\begin{array}{c|c|c|c}
d & 1 & 2 & 3 \\
   \hline
 |\mathbb{X}| & 6 & 6 & 6
 \\ 
   \hline
 H_\mathbb{X}(d)    \    & 3 & 5 &6 
 \\   
   \hline
 \delta_{\mathbb{X}}(d) &3& 2& 1 \\ 
\hline
 {\rm fp}_{I(\mathbb{X})}(d) &3& 2& 1 \\ 
\end{array}
\right.
\end{eqnarray*}

By Corollary~\ref{jan1-16}, if the hyperplanes
$H_1,\ldots,H_d$ in $\mathbb{P}^{2}$ avoid a point $[P]$ in
$\mathbb{X}$ but otherwise cover all the other $|\mathbb{X}|-1$ points
of $\mathbb{X}$, then $d\geq{\rm reg}(S/I(\mathbb{X}))=3$.
\end{example}

Next we give an example of a graded vanishing ideal over a finite field, which
is not Geil--Carvalho, by computing all possible initial ideals. 

\begin{example}\label{not-geil-carvalho}
Let $\mathbb{X}=\mathbb{P}^2$ be the projective space over the field
$\mathbb{F}_2$ and let $I=I(\mathbb{X})$ be the vanishing
ideal of $\mathbb{X}$. Using the procedure below in {\it
Macaulay\/}$2$ \cite{mac2} we get that the binomials  
$t_1t_2^2-t_1^2t_2,\, t_1t_3^2-t_1^2t_3,\, t_2t_3^2-t_2^2t_3$ 
form a universal Gr\"obner basis of $I$, that is, they form a Gr\"obner basis
for any monomial order. The ideal $I$ has exactly six different
initial ideals and $\delta_\mathbb{X}\neq {\rm fp}_I$ for each of
them, that is, $I$ is not a Geil--Carvalho ideal. 
The basic parameters of the projective 
Reed-Muller code $C_\mathbb{X}(d)$ are shown in the following
table.
\begin{eqnarray*}
\hspace{-11mm}&&\left.
\begin{array}{c|c|c|c}
d & 1 & 2 & 3 \\
   \hline
 |\mathbb{X}| & 7 & 7 & 7 
 \\ 
   \hline
 H_\mathbb{X}(d)    \    & 3 & 6 &7 
 \\   
   \hline
 \delta_{\mathbb{X}}(d) &4& 2& 1\\ 
\hline
 {\rm fp}_{I(\mathbb{X})}(d) &4& 1& 1\\ 
\end{array}
\right.
\end{eqnarray*}
\begin{verbatim}
load "gfaninterface.m2"
S=ZZ/2[symbol t1, symbol t2, symbol t3]
I=ideal(t1*t2^2-t1^2*t2,t1*t3^2-t1^2*t3,t2*t3^2-t2^2*t3)
universalGroebnerBasis(I)
(InL,L)= gfan I, #InL
init=ideal(InL_0)
M=coker gens gb init
f=(x)-> if not quotient(init,x)==init then degree ideal(init,x) else 0
fp=(d) ->degree M -max apply(flatten entries basis(d,M),f)
apply(1..regularity(M),fp)
\end{verbatim}
\end{example}

\smallskip

\noindent {\bf Acknowledgments.} We thank the referee for a careful
reading of the paper and for the improvements suggested. 

\bibliographystyle{plain}

\begin{thebibliography}{10}

\bibitem{alon-furedi} N. Alon and Z. F\"uredi, Covering the cube by
affine hyperplanes, European J. Combin. {\bf 14} (1993), 79--83. 

\bibitem{AM}{M.~F. Atiyah and I.~G. Macdonald, {\it Introduction to 
Commutative Algebra}, Addison-Wesley, Reading, MA, 1969.}

\bibitem{clark} A. Bishnoi, P. L. Clark, A. Potukuchi and J. R.
Schmitt, On zeros of a polynomial in a finite grid, Combin. Probab.
Comput., to appear.

\bibitem{clark1} A. Bishnoi, P. L. Clark, A. Potukuchi and J. R.
Schmitt, On the Alon-F\"uredi bound, Electron. 
Notes Discrete Math. {\bf 54} (2016), 57--62.

\bibitem{carvalho} C.  Carvalho, On the second Hamming weight of some
Reed-Muller type codes, Finite Fields Appl. {\bf 24} (2013), 
88--94.

\bibitem{Chardin} M. Chardin and G. Moreno-Soc\'\i as, Regularity of
lex-segment ideals: 
some closed formulas and
applications, Proc. Amer. Math. Soc. {\bf 131} (2003), no. 4,
1093--1102 (electronic). 

\bibitem{CNR} CoCoATeam, CoCoA: a system for doing Computations
 in Commutative Algebra. Available at {\tt 
 http://cocoa.dima.unige.it}.

\bibitem{CLO} D. Cox, J. Little and D. O'Shea, {\it Ideals, 
Varieties, and Algorithms\/}, Springer-Verlag, 1992.

\bibitem{delsarte-goethals-macwilliams} 
P. Delsarte, J. M. Goethals and F. J. MacWilliams, 
On generalized Reed-Muller codes and their relatives, 
Information and Control {\bf 16} (1970), 403--442.

\bibitem{duursma-renteria-tapia} I. M. Duursma, C. Renter\'\i a and
H. Tapia-Recillas,  
Reed-Muller codes on complete intersections, Appl. Algebra Engrg.
Comm. Comput.  {\bf 11}  (2001),  no. 6, 455--462.

\bibitem{Eisen}{D. Eisenbud, {\it Commutative Algebra with a view
toward Algebraic Geometry\/}, Graduate
Texts in  Mathematics {\bf 150}, Springer-Verlag, 1995.}

\bibitem{fulton} W. Fulton, {\it Algebraic curves\/}, 
Advanced Book Classics, Addison-Wesley Publishing Company, Advanced
Book Program, Redwood City, CA, 1989, An introduction to algebraic
geometry, Notes written with the collaboration of Richard Weiss,
Reprint of 1969 original.

\bibitem{geil} O. Geil, 
On the second weight of generalized Reed-Muller codes, Des. Codes
Cryptogr. {\bf 48} (2008), 323--330. 

\bibitem{geil-thomsen} O. Geil and C. Thomsen, 
Weighted Reed--Muller codes revisited, Des. Codes Cryptogr. {\bf 66}
(2013), 195--220.

\bibitem{gold-little-schenck} L. Gold, J. Little and H. Schenck, 
Cayley-Bacharach and evaluation codes on complete intersections, 
J. Pure Appl. Algebra {\bf 196} (2005), no. 1, 91--99. 

\bibitem{GRT} M. Gonz\'alez-Sarabia, C. Renter\'\i a and H.
Tapia-Recillas, Reed-Muller-type codes over the Segre variety,  
Finite Fields Appl. {\bf 8}  (2002),  no. 4, 511--518. 

\bibitem{mac2} D. Grayson and M. Stillman, 
{\em Macaulay\/}$2$, 1996. 
Available via anonymous ftp from {\tt math.uiuc.edu}.

\bibitem{singular} G. M. Greuel and G. Pfister, {\it A Singular
Introduction to 
Commutative Algebra}, 2nd extended edition, Springer, Berlin, 2008. 

\bibitem{harris} J. Harris, {\it Algebraic Geometry. A first course},
Graduate Texts in Mathematics {\bf 133},  Springer-Verlag, New York, 1992.

\bibitem{cocoa-book} M. Kreuzer and L. Robbiano, {\it Computational
Commutative Algebra} 2, Springer-Verlag, Berlin, 2005. 

\bibitem{cartesian-codes} H. H. L\'opez, C. Renter\'\i a and R. H.
Villarreal, Affine cartesian codes,
Des. Codes Cryptogr. {\bf 71} (2014), no. 1, 5--19.

\bibitem{affine-codes} H. H. L\'opez, E. Sarmiento, M. Vaz Pinto and 
R. H. Villarreal, Parameterized affine codes, Studia Sci. Math.
Hungar. {\bf 49} (2012), no. 3, 406--418.

\bibitem{MacWilliams-Sloane} F. J. MacWilliams and N. J. A. Sloane, 
The Theory of Error-correcting Codes, North-Holland, 1977. 

\bibitem{hilbert-min-dis} J. Mart\'\i nez-Bernal, Y. Pitones 
and R. H. Villarreal, Minimum
distance functions of graded ideals   
and Reed-Muller-type codes, 
J. Pure Appl. Algebra {\bf 221} (2017), 251--275. 

\bibitem{Migliore} J. C. Migliore, {\it Introduction to liaison
theory and Deficiency Modules\/}, Progress in Mathematics {\bf 165},
Birkh\"auser Boston, Inc., Boston, MA, 1998.


\bibitem{algcodes} C. Renter\'\i a, A. Simis and R. H. Villarreal,
Algebraic methods for parameterized codes 
and invariants of vanishing ideals over finite fields, Finite Fields
Appl. {\bf 17} (2011), no. 1, 81--104.  

\bibitem{sage} SAGE Mathematical Software, http://www.sagemath.org.

\bibitem{escalier} M. Sala, T. Mora, L. Perret, S. Sakata
and C. Traverso (eds.), {\it Gr\"obner Bases, Coding, and
Cryptography\/}, RISC Book Series, Springer, Heidelberg, 2009.


\bibitem{ci-codes} E. Sarmiento, M. Vaz Pinto and R. H. Villarreal, 
The minimum distance of parameterized codes on projective tori, 
Appl. Algebra Engrg. Comm. Comput. {\bf 22} (2011), no. 4, 249--264.

\bibitem{schmidt} W. M. Schmidt, {\it Equations over finite fields, An elementary
approach},  Lecture Notes in Mathematics {\bf 536}, Springer-Verlag, Berlin-New York, 1976. 

\bibitem{sorensen} A. S{\o}rensen, Projective Reed-Muller codes, 
IEEE Trans. Inform. Theory {\bf 37} (1991), no. 6, 1567--1576.

\bibitem{Sta1}{R. Stanley, Hilbert functions of graded 
algebras, Adv.
Math. {\bf 28} (1978), 57--83.}

\bibitem{tohaneanu-vantuyl} S. Toh\v{a}neanu and A. Van Tuyl, 
Bounding invariants of fat points using a coding theory construction,
J. Pure Appl. Algebra {\bf 217} (2013), no. 2, 269--279. 

\bibitem{tsfasman} M. Tsfasman, S. Vladut and D. Nogin, {\it
Algebraic 
geometric codes{\rm:} basic notions}, Mathematical Surveys and
Monographs {\bf 139}, American Mathematical Society, 
Providence, RI, 2007. 

\bibitem{monalg-rev} R. H. Villarreal, {\it Monomial Algebras, Second Edition\/}, 
Monographs and Research Notes in Mathematics, Chapman and Hall/CRC, 2015.

\bibitem{Vogel} W. Vogel, Lectures on results on Bezout's theorem, 
Tata Institute of Fundamental Research Lectures on Mathematics
and Physics { \bf 74}, Springer-Verlag, Berlin, 1984. 

\end{thebibliography}

\end{document}